\documentclass[11pt, reqno]{article}
\usepackage{amsfonts}
\usepackage{amsthm,amsmath,amssymb,amsfonts,mathrsfs,hyperref,microtype}
\usepackage{graphicx,color}

\newcommand{\weq}{\ = \ }

\newtheorem{theorem}{Theorem}[section]
\newtheorem{proposition}{Proposition}[section]

\newtheorem{remark}{Remark}[section]
\newtheorem{lemma}{Lemma}[section]

\def\N{{\rm I\kern-0.16em N}}
\def\R{{\rm I\kern-0.16em R}}
\def\E{{\rm I\kern-0.16em E}}
\def\Pro{{\rm I\kern-0.16em P}}
\def\F{{\rm I\kern-0.16em F}}
\def\B{{\rm I\kern-0.16em B}}
\def\C{{\rm I\kern-0.46em C}}
\def\G{{\rm I\kern-0.50em G}}

\numberwithin{equation}{section}
\font\eka=cmex10
\usepackage{ae}
\def\ind{\mathrel{\hbox{\rlap{%
\hbox to 7.5pt{\hrulefill}}\raise6.6pt\hbox{\eka\char'167}}}}
\begin{document}
\title{\textbf{Stochastic differential equations with noise perturbations and Wong--Zakai approximation of fractional Brownian motion}}
\date{\today}

\renewcommand{\thefootnote}{\fnsymbol{footnote}}

\author{Lauri Viitasaari\footnotemark[1] \, and \, Caibin Zeng\footnotemark[2]}

\footnotetext[1]{Aalto University School of Business, Department of Information and Service Management, P.O. Box 11000,
FIN-00076 AALTO, Finland, {\tt lauri.viitasaari@iki.fi}.}

\footnotetext[2]{School of Mathematics, South China University of Technology, Guangzhou 510640, China, {\tt macbzeng@scut.edu.cn}.}
\maketitle

\begin{abstract}
In this article we study effects that small perturbations in the noise have to the solution of differential equations driven by H\"older continuous functions of order $H>\frac12$. As an application, we consider stochastic differential equations driven by a fractional Brownian motion. We introduce a Wong--Zakai type stationary approximation to the fractional Brownian motions and prove that it converges in a suitable space. Moreover, we provide sharp results on the rate of convergence in the $p$-norm. Our stationary approximation is suitable for all values of $H\in (0,1)$.
\end{abstract}



\section{Introduction}
\label{sec1}

The philosophy of using appropriate differential equations with more regular drivers to approximate stochastic differential equations dates back to the pioneer work by Wong and Zakai \cite{WZ1965b,WZ1965a}, in which both continuous piecewise linear approximations and piecewise smooth approximations for one-dimensional Brownian motions were proposed. Particularly, the convergence of solutions to the approximating equations was proved in the mean sense and almost surely under some suitable conditions on coefficients, respectively. However, these results are not applicable to high-dimensional cases. For instance, McShane \cite{McShane1972} provided a counter example showing that Wong--Zakai's result does not hold for a two dimensional Brownian motion approximated by smooth functions. Meanwhile, the Wong--Zakai's piecewise linear (polygonal) approximation for high-dimensional Brownian motions was introduced by Stroock and Varadhan in \cite{SV1972}, where also the convergence in law was proved and used to determine the support of diffusion processes. Later on the shift operator was incorporated into approximations of high-dimensional Brownian motion in two different ways by Ikeda et al. \cite{INY1977,IW1992}. In these articles the convergence was established in the mean square sense, uniformly over finite time interval. More recently, Kelly and Melbourne \cite{KM2016} introduced a class of smooth approximations in the integral form by involving a $C^2$ uniformly hyperbolic flow on a compact manifold. In this case the authors established weak convergence towards the limit equation by the methods of rough path theory. However, none of the mentioned three articles \cite{INY1977,IW1992,KM2016} studied the limit equations understood in the Stratonovich sense.

To the best of our knowledge, one of the first stationary smooth approximations is by Lu and Wang \cite{LW2010,LW2011}. They approximated a one-dimensional white noise by a stationary Gaussian process, and they applied the approximation to the study of the chaotic behaviour of randomly forced ordinary differential equations with a homoclinic orbit to a saddle fixed point. Stationary Wong--Zakai approximation of the white noise by Lu and Wang was extended to high-dimensional situations by Shen and Lu in \cite{SL2017}, where it was proved that the solutions of Wong--Zakai approximations converge in the mean square to the solutions of the Stratonovich stochastic differential equation. Afterwards the stationary Wong--Zakai approximations have been applied to investigate complex dynamics of stochastic partial differential equations \cite{LW2017,WLW2018,GLW2019,SZLW2019}. In particular, in \cite{SZLW2019} it was proved that the solutions of Wong--Zakai approximations converge almost surely to the solutions of the Stratonovich stochastic evolution equation. We also wish to mention articles \cite{Konecny1983,Protter1985,Nakao1986,KP1991} generalising the Wong--Zakai approximations to the case of stochastic differential equations driven by martingales and semimartingales \cite{Konecny1983,Protter1985,Nakao1986,KP1991}. 


The interaction between the (small) system and its (large) environment is described in general by a stochastic force which can be coloured or white, Gaussian or non-Gaussian, Markovian or non-Markovian, and semimartingale or non-semimartingale. In general, there is no \emph{a priori} reason to require the noise to be independent in the disjoint time intervals, and thus the modeller is forced to drop the Markov and martingale properties of the noise. 

Maybe the best known and studied generalisation of the Brownian motion into non-Markovian and non-semimartingale setting is the fractional Brownian motion (fBm). Fractional Brownian motion shares many desired properties with the Brownian motion such as continuity, scale-invariance, Gaussianity, and stationarity of the increments. However, fBm does not have independent increments. Actually, fBm may exhibit long or short memory, depending on the so-called Hurst parameter $H\in(0,1)$. It is also well known that one cannot reduce fBm to
a Markovian situation without adding infinitely many degrees of freedom. On the other hand, being relatively simple and able to capture memory effects simultaneously makes fBm an interesting model for many phenomena. For these reasons there has been an increasing interest in the literature on stochastic (partial) differential equations driven by the (cylindrical) fBm. For detailed discussions with practical applications of the fBm, we refer to monographs \cite{ST1994,BHOZ2008,Mishura2008} and references therein.

The seminal papers \cite{WZ1965b,WZ1965a} on Wong--Zakai approximations in the case of standard Brownian motion were published over fifty years ago, during the same time when first studies on fBm appeared. However, still most of the literature on Wong--Zakai approximations, including all of the previously mentioned articles, have focused on the Markovian and/or semimartingale settings. On the Wong--Zakai approximations in the case of fBm, one of the first published articles by Tudor \cite{Tudor2009} appeared in 2009. Tudor studied Wong--Zakai type approximations and convergence issue on a class of It\^{o}-Volterra equations driven by an fBm with $1/2<H<1$, using an integral representation for fBm and piecewise linear approximation of Brownian motion. Later on, the Galerkin approximation was applied to fractional noise by Cao et. al. \cite{CHL2017} who acquired an optimal rate of convergence in the case of space-time fractional white noise, with Hurst index $0<H<1/2$. Moreover, the dyadic polygonal approximations of fBm have been studied recently from the viewpoint of the rough path theory. The first result in this direction was presented by Coutin and Qian \cite{CZ2000}, reporting a Wong--Zakai type approximation theorem for solutions of stochastic differential equations driven by fBm with $1/4<H<1/2$. Hu and Nualart \cite{HN2009} derived, by using fractional calculus, the rate of convergence for the Wong--Zakai approximation for fBm with $1/3<H<1/2$. However, these two results do not provide sharp estimates. Refined results on the rate of convergence were obtained by Friz-Riedel \cite{FR2014} in the pathwise sense, and by Bayer et al. \cite{BFRS2016} in the probabilistic sense, respectively. Finally, optimal upper bound for the error of the Wong--Zakai approximation for fBm with $1/4<H<1/2$ was provided by Naganuma \cite{Naganuma2016}. However, to the best of our knowledge, the stationary Wong--Zakai approximations and results on the rate of convergence results for the fractional noise and corresponding differential equations are not yet studied in the literature.

In this article we fill this gap in the literature and introduce stationary Wong--Zakai type approximations for the fBm in the full range $H\in(0,1)$. Moreover, we derive sharp results on the rate of convergence in the $p$:th moment sense, valid for any $p\geq 1$. As an application we prove that, in the regime $H>\frac12$, solutions to the differential equations driven by the approximation converge to the solution to the original stochastic differential equation driven by the fBm itself. We will focus on the case of fBm in which we can apply some fine known properties to obtain exact rate of convergence. However, we stress that if one is simply interested on the convergence but not on the exact rate, our method and many of our results can be easily extended to cover a very large class of stochastic processes. This class includes all H\"older continuous Gaussian processes and also H\"older continuous processes living in some fixed chaos. In fact, our approximation converge under very mild conditions on the process, and the approximation for the process is stationary as long as the process has stationary increments. In the context of stochastic differential equations, we prove that as long as the true noise are approximated properly in certain Besov type space, then the corresponding solutions converge as well. This justifies the claim presented in the abstract. 

The rest of the article is organised as follows. In Section \ref{sec2} we recall some preliminaries needed for our analysis. In particular, we recall the concept of generalised Lebesgue-Stieltjes integrals together with some useful inequalities. In Section \ref{sec3} we introduce a stationary Wong--Zakai approximation of the fractional noise and study its convergence on the full range $H\in (0,1)$. In Section \ref{sec4} we study effects of noise perturbations to the solutions of differential equations driven by H\"older signals of order $\alpha>\frac12$. There we also apply results of Section \ref{sec3}, on range $H\in\left(\frac12,1\right)$, to study Wong--Zakai approximations of stochastic differential equations driven by fractional Brownian motions. We end the paper with a short discussion on the generality of our results.

\section{Preliminaries}
\label{sec2}

In this section we discuss briefly some preliminaries we need for our analysis. In particular, we recall the notion of generalised Lebesgue--Stieltjes integrals and some \emph{a priori} estimates. 
For more details, we refer to the article \cite{NR2002} and a monograph \cite{SKM1993}.

Let $(a,b)$ be a nonempty bounded interval. For given $p \in [1,\infty]$ we use the usual notation $L_p =L_p(a,b)$ to denote $p$-integrable functions, or essentially bounded in the case $p=\infty$. The fractional left and right Riemann--Liouville integrals of order $\alpha > 0$ of a function $f \in L_1$ are given by
\[
 I^\alpha_{a+}f(t)
 \weq \frac{1}{\Gamma(\alpha)} \int_a^t \frac{f(s)}{(t-s)^{1-\alpha}} \, ds
\]
and
\[
 I^\alpha_{b-}f(t)
 \weq \frac{(-1)^{-\alpha}}{\Gamma(\alpha)} \int_t^b \frac{f(s)}{(t-s)^{1-\alpha}} \, ds.
\]
It is known that the above integrals converge for almost every $t \in (a,b)$, and $I^\alpha_{a+} f$ and $I^\alpha_{b-}$ may be considered as functions in $L_1$. By convention, $I^0_{a+}$ and $I^0_{b-}$ are defined as the identity operator. Moreover, the integral operators $I^\alpha_{a+}, I^\alpha_{b-}: L_1 \to L_1$ are linear and injective. The inverse operators are called Riemann--Liouville fractional derivatives, denoted by $I^{-\alpha}_{a+} = (I^\alpha_{a+})^{-1}$ and $I^{-\alpha}_{b-} = (I^\alpha_{b-})^{-1}$, respectively.

For any $\alpha \in (0,1)$ and for any $f \in I^\alpha_{a+}(L_1)$ and $g \in I^\alpha_{b-}(L_1)$, the Weyl--Marchaud derivatives are defined by formulas
\[
 D_{a+}^\alpha f(t)
 \weq \frac{1}{\Gamma(1-\alpha)}\left( \frac{f(t)}{(t-a)^\alpha} + \alpha \int_a^t \frac{f(t)-f(s)}{(t-s)^{\alpha+1}} \, ds \right)
\]
and
\[
 D_{b-}^\alpha g(t)
 \weq \frac{(-1)^\alpha}{\Gamma(1-\alpha)}\left( \frac{g(t)}{(b-t)^\alpha} + \alpha \int_t^b \frac{g(t)-g(s)}{(s-t)^{\alpha+1}} \, ds \right).
\]
These are well-defined and they coincide with the Riemann--Liouville derivatives according to $D_{a+}^\alpha f(t) = I^{-\alpha}_{a+} f(t)$ and $D_{b-}^\alpha g(t) = I^{-\alpha}_{b-} g(t)$, where equalities hold for almost every $t \in (a,b)$.

We are now ready to recall the concept of generalised Lebesgue--Stieltjes integrals. For this, let $f$ and $g$ be functions such that the limits $f(a+), g(a+), g(b-)$ exist in $\R$, and denote $f_{a+}(t) = f(t) - f(a+)$ and $g_{b-}(t) = g(t) - g(b-)$. Suppose now that $f_{a+} \in I^\alpha_{a+}(L_p)$ and $g_{b-} \in I^{1-\alpha}_{b-}(L_q)$ for some $\alpha \in [0,1]$ and $p,q \in [1,\infty]$ such that $1/p+1/q =1$.
In this case the generalised Lebesgue-Stieltjes integral
\begin{equation}
 \label{eq:ZSIntegral}
 \begin{aligned}
 \int_a^b f_t \, dg_t
 &\weq (-1)^\alpha \int_a^b D^{\alpha}_{a+} (f-f(a+))(t) \, D^{1-\alpha}_{b-} (g-g(b-))(t) \, dt \\
 &\qquad \qquad + f(a+)(g(b-)-g(a+)),
 \end{aligned}
\end{equation}
is well-defined. Moreover, it can be proved that the right side does not depend on $\alpha$.

We follow closely the approach of \cite{NR2002}, and for this we need to introduce some necessary spaces and norms.  For $\beta\in\left(0,\frac12\right)$, we denote by $W_{1,1-\beta}(0,T;\mathbb{R}^m)$ the space of measurable functions with values in $\mathbb{R}^m$ equipped with a norm
$$
\Vert f\Vert_{1,1-\beta} = \sup_{0<s<t<T}\left[ \frac{|f(t) - f(s)|}{(t-s)^{1-\beta}} + \int_s^t \frac{|f(y)-f(s)|}{|y-s|^{2-\beta}} \, dy\right].
$$
Here and below, $|\cdot|$ denotes the Euclidean norm. Similarly, we use a norm
$$
\Vert f \Vert_{2,\beta} = \int_0^T \frac{|f(s)|}{s^\beta}ds + \int_0^T \int_0^t \frac{|f(t)-f(s)|}{|t-s|^{\beta+1}}ds dt.
$$
We also consider the function space $W_{\beta,\infty}(0,T;\mathbb{R}^m)$,  the space of measurable functions with values in $\mathbb{R}^m$ with a norm
$$
\Vert f\Vert_{\beta,\infty} =\sup_{t\in[0,T]} \left[|f(t)|+\int_0^t \frac{|f(t)-f(s)|}{|t-s|^{\beta+1}}ds\right].
$$
Clearly, we have
$$
\Vert f\Vert_{2,\beta} \leq \frac{T^{1-\beta}\max(1,T^\beta)}{1-\beta}\Vert f\Vert_{\beta,\infty},
$$
which we will use in the sequel.
Similarly, we denote by $W_{\beta,\infty}(0,T;\mathbb{R}^{n\times m})$ the space of measurable vector functions $f=(f_1, \cdots, f_n)$ with $f_i\in(0,T;\mathbb{R}^m)$ equipped with a norm
$$
\Vert f\Vert_{\beta,\infty}=\max_{i=1, \cdots, n}\Vert f_i\Vert_{\beta,\infty}.
$$
We also denote by $C_{\beta}$ the space of $\beta$-H\"older continuous functions.

Finally, we recall that if $f\in W_{\beta,\infty}(0,T;\mathbb{R}^{n\times m})$ and $g\in W_{1,1-\beta}(0,T;\mathbb{R}^m)$, then the integral $\int_0^t f_sdg_s$ exists for all $t$. Moreover, the integral belongs to $C_{1-\beta}(0,T;\mathbb{R}^n) \subset W_{\beta,\infty}(0,T;\mathbb{R}^n)$. Actually, by Proposition 4.1 of \cite{NR2002} we have the estimate
$$
\Vert \int_0^\cdot f_sdg_s \Vert_{1-\beta} \leq C\Vert g\Vert_{1,1-\beta}\Vert f\Vert_{\beta,\infty}.
$$
This yields immediately a bound
\begin{equation}
\label{eq:Holder-for-integral}
\left|\int_s^t f_udg_u \right| \leq C|t-s|^{1-\beta}\Vert g\Vert_{1,1-\beta}\Vert f\Vert_{\beta,\infty}.
\end{equation}
Moreover, the integral satisfies also a bound
\begin{equation}
\label{eq:basic-bound}
\left|\int_0^t f_s dg_s\right| \leq C\Vert g\Vert_{1,1-\beta}\Vert f\Vert_{2,\beta}
\end{equation}
which we will use throughout the paper. We also recall the following Gronwall type lemma that is needed for the proof of Theorem \ref{thm4.1}.
\begin{lemma}\cite[Lemma 7.6]{NR2002}
\label{lemma4.1}
Fix $0\leq\alpha<1$, $a, b\geq0$. Let $x:[0,\infty)\rightarrow[0,\infty)$ be a continuous function such that for each $t$ we have
\begin{equation*}
x_t\leq a+bt^\alpha\int_0^t(t-s)^{-\alpha}s^{-\alpha}x_sds.
\end{equation*}
Then
\begin{equation*}
x_t\leq a\Gamma(1-\alpha)\sum_{n=0}^\infty\frac{(b\Gamma(1-\alpha)t^{1-\alpha})^n}{\Gamma[(n+1)(1-\alpha)]}
\leq\ a d_\alpha\exp[c_\alpha tb^{\frac{1}{1-\alpha}}],
\end{equation*}
where $c_\alpha$ and $d_\alpha$ are positive constants depending only on $\alpha$.
\end{lemma}

\section{Wong--Zakai approximations for the fractional noise}
\label{sec3}
Let $W^H(t)$ be an $m$-dimensional fractional Brownian motion with Hurst index $H\in(0,1)$ on a complete probability space $(\Omega,\mathcal{F},\mathbb{P}^H)$, i.e. the components $W^{H,j}(t)$, $j=1,2,\ldots,m$, of $W^H(t)$ are independent centered Gaussian processes with the covariance function
\begin{equation}
\label{eq:fbm-covariance}
  R_H(s,t)=\frac{1}{2}
  \left(s^{2H}+t^{2H}-|t-s|^{2H}\right)~~\text{for~} s,t\in \mathbb{R}^{+}.
\end{equation}
Throughout the article, we denote by $L_p(\Omega)$ the space of random variables with $p$th moment finite. In the sequel, we will work with the canonical version of the fBm. In fact, let $\Omega=C_{0}(\mathbb{R};\mathbb{R}^{m})$ be the space of continuous paths  with values zero at zero equipped with the compact open topology. Let also $\mathcal{F}$ be defined as the Borel-$\sigma$-algebra and let $\mathbb{P}_H$ be the distribution of $B^H(t)$. We consider the Wiener shift given by
\begin{equation*}
  \theta_{t}\omega(\cdot)=\omega(\cdot+t)-\omega(t),
\end{equation*}
where $t\in \mathbb{R}$ and $\omega\in C_{0}(\mathbb{R};\mathbb{R}^{m})$. It follows from \cite[Theorem 1]{GS2011} that the quadruple $(\Omega,\mathcal{F},\mathbb{P}^H,\theta)$ is an ergodic metric dynamical system. In the sequel, we will identify $B^H(\cdot,\omega)$ with the continuous path $\omega(\cdot)$. For each $\delta>0$, we define the random variable $\mathcal {G}_\delta:\Omega\to\mathbb{R}^m$ by
\begin{equation*}
  \mathcal {G}_\delta(\omega)=\frac{1}{\delta}\omega(\delta).
\end{equation*}
Then 
\begin{equation}\label{equ4.1}
  \mathcal{G}_\delta(\theta_{t}\omega)=\frac{1}{\delta}(\omega(\delta+t)-\omega(t)),
\end{equation}
and it follows from stationarity of the increments of fBm (see, e.g. \cite{Mishura2008,BHOZ2008}) that $\mathcal{G}_\delta(\theta_{t}\omega)$ is a stationary Gaussian process. Let
\begin{equation*}
C_\omega=\sup_{s\in\mathbb{Q}}\frac{|\omega(s)|}{|s|+1}.
\end{equation*}
Since 
\begin{equation*}
\lim_{s\rightarrow\pm\infty}\frac{|\omega(s)|}{|s|}=0,
\end{equation*}
we also get
\begin{equation*}
|\omega(s)|\leq C_\omega(|s|+1).
\end{equation*}
Together with $\theta_t\omega(s)=\omega(s+t)-\omega(t)$, this gives
\begin{equation*}
C_{\theta_t\omega}\leq 2C_\omega(|t|+1).
\end{equation*}
In particular, we have
\begin{equation}\label{equ3.3}
|\mathcal {G}_\delta(\theta_t\omega)|\leq\frac{2}{\delta}(\delta+1)C_\omega(|t|+1)
\end{equation}
which will be used later on.

Next we show that $\mathcal{G}_\delta(\theta_{t}\omega)$ can be viewed as an approximation of fractional white noise in the Wong--Zakai sense. The following result provides us the first estimate on the convergence rate in the mean sense. In Proposition \ref{pro:exact-L2-rate} we will apply the properties of the fBm to provide the exact rate of convergence. However, we wish to present the following simple proof that can be adapted easily to more general cases than fBm (see also Remark \ref{remark:generalisation} and discussions in Section \ref{sec:discussions}). 
\begin{lemma}
\label{lma:L2-convergence}
We have
$$
\lim_{\delta\to0^{+}}\sup_{t\in[0,T]}\left|\int^t_0\mathcal{G}_\delta(\theta_{s}
\omega)ds-\omega(t)\right|=0.
$$
Moreover, for any $p\geq 1$ there exists a constant $C>0$ such that for any $t\in[0,T]$, we have
  \begin{equation*}
\E\left|\int^t_0\mathcal{G}_\delta(\theta_{s}
\omega)ds-\omega(t)\right|^p \leq C\delta^{Hp}.
\end{equation*}
\end{lemma}
\begin{proof}
For each $T\in\mathbb{R}$ and $0<t<T$ we have
\begin{equation*}
  \int^t_0\mathcal{G}_\delta(\theta_{s}\omega)ds
  =\int_0^t \frac{\omega(\delta+s)-\omega(s)}{\delta}ds
  =\left(\int^{t+\delta}_{\delta}-\int^t_0\right)
  \frac{\omega(s)}{\delta}ds
=\left(\int^{t+\delta}_t-\int^{\delta}_0\right)
\frac{\omega(s)}{\delta}ds.
\end{equation*}
Consequently,
\begin{eqnarray*}
\left|\int^t_0\mathcal{G}_\delta(\theta_{s}\omega)ds-\omega(t)\right|\leq
\left|\int^{t+\delta}_t\frac{\omega(s)-\omega(t)}{\delta}ds\right|
+\left|\int^{\delta}_0\frac{\omega(s)}{\delta}ds\right|.
\end{eqnarray*}
Hence, by using the uniform continuity of $\omega$ on the interval $[0,T+\delta]$, we obtain
\begin{equation*}
\lim_{\delta\rightarrow0^+} \sup_{t\in[0,T]}\left|\int^t_0\mathcal{G}_\delta(\theta_{s}\omega)ds
-\omega(t)\right|=0
\end{equation*}
that proves the first claim. For the second claim, basic inequality
$$
\sqrt{x_1^2 + \ldots x_m^2} \leq |x_1| + \ldots |x_n|
$$
together with Minkowski inequality implies the claim once we have proved the claim for $m=1$, i.e. in the one-dimensional case. This now follows easily. Indeed, by Gaussianity we have
\begin{equation*}
\E\left|\int^t_0\mathcal{G}_\delta(\theta_{s}\omega)ds-
\omega(t)\right|^p = C_p \left[\E\left|\int^t_0\mathcal{G}_\delta(\theta_{s}\omega)ds-
\omega(t)\right|\right]^p,
\end{equation*}
where
\begin{equation*}
\begin{split}
\E\left|\int^t_0\mathcal{G}_\delta(\theta_{s}\omega)ds-
\omega(t)\right| &\leq \frac{1}{\delta}\int_t^{t+\delta}\E|\omega(s)-\omega(t)|ds + \frac{1}{\delta}\int_0^\delta \E|\omega(s)|ds\\
&\leq \frac{1}{\delta}\int_t^{t+\delta}|s-t|^H ds + \frac{1}{\delta}\int_0^\delta s^Hds \\
&=\frac{2}{H+1}\delta^H.
\end{split}
\end{equation*}
This concludes the proof.
\end{proof}
\begin{remark}
\label{remark:generalisation}
We remark that the above statement remains valid for a much larger class of processes. Actually, we have only used continuity and hypercontractivity $\E|X_t|^p \leq C_p [\E|X_t|]^p$, valid for Gaussian processes $X$, together with $\E|X_t-X_s| \leq C|t-s|^H$. Thus the above result easily extends to arbitrary Gaussian process with H\"older continuous paths (see, e.g. \cite{Azmoodeh-Sottinen-Viitasaari-Yazigi-2014}) and beyond. On the other hand, in particular cases such as fBm, one can refine the arguments and obtain exact rate of convergence, see Proposition \ref{pro:exact-L2-rate} and Theorem \ref{theorem:WZ-fbm-convergence} below. For detailed discussion on the generalisations, see Section \ref{sec:discussions}.
\end{remark}
In order to establish convergence of the solutions of the approximating stochastic differential equations we need convergence in the space $W_{1,1-\beta}$. This on the other hand can be achieved by proving convergence in the space $L_p(\Omega)$. We start with a result stating the exact error in $L_p(\Omega)$ in the one-dimensional case, which may also have its own interest.

By notation $f = o(g)$ we mean standard Landau's notation indicating $\frac{|f|}{|g|} \to 0$. For a fixed $H\in(0,1)$ we also introduce a function $\Theta : (0,\infty) \mapsto (0,\infty)$ defined by
\begin{equation}
\label{eq:theta}
\Theta(x) = \int_0^1\int_0^1 x^{2H}+2u^{2H}+|s-u+x|^{2H}-(u+x)^{2H}-|u-x|^{2H}-|u-s|^{2H}dsdu.
\end{equation}
It turns out that $\Theta(x)$ plays a crucial role in the rate of convergence. We also note that while here the integral could be computed explicitly, the form \eqref{eq:theta} is the most convenient for our proofs. We begin with the following lemma that studies the properties of $\Theta(x)$. 
\begin{lemma}
\label{lma:theta-properties}
Let $H\in (0,1)$ and let $\Theta(x)$ be given by \eqref{eq:theta}. Then
\begin{enumerate}
\item we have $\Theta(x) = \frac{1}{H+1} + o(1)$ as $x\to \infty$, and
\item we have $\Theta(x) =\frac{H(2H+3)}{(2H+1)(H+1)}x^{2H} + o(x^{2H})$ as $x\to 0$.
\end{enumerate}
\end{lemma}
\begin{proof}
We begin with the first claim concerning asymptotic behaviour as $x\to \infty$. For any $x>1$ we have
$$
\Theta(x) = \int_0^1\int_0^1 x^{2H}+2u^{2H}+(x+u-s)^{2H}-(x+u)^{2H}-(x-u)^{2H}-|u-s|^{2H}dsdu.
$$
Moreover, by Taylor approximation we have, for any $a\in[-1,1]$,
$$
(x+a)^{2H} = x^{2H} + 2Hx^{2H-1}a + H(2H-1)x^{2H-2}a^2 + R(x),
$$
where the remainder satisfies $R(x) = o(x^{2H-2})$ as $x\to \infty$, uniformly in $a\in[-1,1]$. Applying this to functions $(x+u-s)^{2H}$, $(x+u)^{2H}$, and $(x-u)^{2H}$ gives us
\begin{eqnarray*}
&&\int_0^1 \int_0^1 x^{2H}+(x+u-s)^{2H}-(x+u)^{2H}-(x-u)^{2H}dsdu \\
&=& \int_0^1 \int_0^1 2Hx^{2H-1}(u-s)+H(2H-1)x^{2H-2}\left[(u-s)^2-u^2-u^2\right]ds du + o(x^{2H-2})\\
&=& H(2H-1)x^{2H-2} \int_0^1 \int_0^1 \left(s^2-u^2 - 2us\right) ds du + o(x^{2H-2}) \\
&= & \frac{H(2H-1)}{2}x^{2H-2} + o(x^{2H-2}).
\end{eqnarray*}
Since $x^{2H-2} = o(1)$ as $x\to\infty$, we obtain 
$$
\Theta(x) = \int_0^1 \int_0^1 2u^{2H} - |u-s|^{2H}dsdu + o(1)= \frac{1}{H+1} + o(1)
$$
proving the first claim. For the second claim concerning asymptotic behaviour as $x\to 0$, let $x<1$. It suffices to prove that 
\begin{equation}
\label{eq:negligible}
\int_0^1 2u^{2H} - (u+x)^{2H}-|u-x|^{2H}du = o(x^{2H}),
\end{equation}
and
\begin{equation}
\label{eq:contributing}
\int_0^1\int_0^1 |u-s+x|^{2H}-|u-s|^{2H}dsdu = -\frac{x^{2H}}{(2H+1)(H+1)} + o(x^{2H}).
\end{equation}
Indeed, combining \eqref{eq:negligible} and \eqref{eq:contributing} leads to 
$$
\Theta(x) = x^{2H} - \frac{1}{(2H+1)(H+1)}x^{2H} + o(x^{2H}) = \frac{H(2H+3)}{(2H+1)(H+1)}x^{2H} + o(x^{2H}).
$$
We begin by showing \eqref{eq:negligible}. We compute
\begin{eqnarray*}
&&\int_0^1 2u^{2H} - (u+x)^{2H}-|u-x|^{2H}du \\
&=& \int_0^1 |u|^{2H}du - \int_x^{x+1}|u|^{2H}du + \int_0^1 |u|^{2H}du - \int_{-x}^{1-x}|u|^{2H}du\\
&=& \left[\int_0^x -\int_1^{x+1} + \int_{1-x}^1 - \int_{-x}^0\right] |u|^{2H}du\\
&=& \left[\int_1^{x+1} - \int_{1-x}^1\right] u^{2H}du.
\end{eqnarray*}
Hence, by applying Taylor's theorem and L'hopital's rule, we can conclude 
$$
\lim_{x\to 0} \frac{\left[\int_1^{x+1} - \int_{1-x}^1\right] u^{2H}du}{x^{2}} = \lim_{x\to 0}\frac{(1+x)^{2H} - (1-x)^{2H}}{2x} = 2H.
$$
This implies 
$$
\int_0^1 2u^{2H} - (u+x)^{2H}-|u-x|^{2H}du = 2Hx^2 + o(x^{2})
$$
and, in particular, that \eqref{eq:negligible} holds.
For \eqref{eq:contributing}, we can compute as above to get
\begin{eqnarray*}
&&\int_0^1\int_0^1 |u-s+x|^{2H}-|u-s|^{2H}dsdu \\
&=& \left[\int_x^{x+1} - \int_0^1\right]\int_0^1 |u-s|^{2H} dsdu \\
&=& \left[\int_1^{x+1} - \int_0^x\right]\int_0^1 |u-s|^{2H} dsdu\\
&=& \left[\int_1^{x+1}\int_0^1 - \int_0^x\int_x^{1} - \int_0^x\int_0^x\right]|u-s|^{2H}dsdu.
\end{eqnarray*}
Direct calculations show that
$$
\int_1^{x+1}\int_0^1|u-s|^{2H}dsdu = \frac{1}{(2H+1)(2H+2)}\left[(x+1)^{2H+2}-1-x^{2H+2}\right]
$$
and
$$
\int_0^x\int_x^{1}|u-s|^{2H}dsdu =\frac{1}{(2H+1)(2H+2)}\left[1-x^{2H+2}-(1-x)^{2H-2}\right].
$$
Hence, by using Taylor approximation again, we obtain
\begin{eqnarray*}
&&\left[\int_1^{x+1}\int_0^1 - \int_0^x\int_x^{1}\right]|u-s|^{2H}dsdu \\
&=& \frac{1}{(2H+1)(2H+2)}\left[(1+x)^{2H+2}+(1-x)^{2H+2}-2\right] \\
& = & x^{2} + o(x^2) = o(x^{2H}).
\end{eqnarray*}
Observing 
$$
\int_0^x \int_0^x |u-s|^{2H}dsdu = x^{2H} \int_0^1 \int_0^1 |u-s|^{2H}dsdu = \frac{x^{2H}}{(2H+1)(H+1)}
$$
shows \eqref{eq:contributing} and concludes the whole proof.
\end{proof}
\begin{proposition}
\label{pro:exact-L2-rate}
Let $H\in (0,1)$ be fixed and let the number of dimensions $m=1$. Then the error process $\left(\int^t_0\mathcal{G}_\delta(\theta_{s}
\omega)ds-\omega(t)\right)_{t\geq 0}$ has stationary increments, and for any $p\geq 1$ and any $t\geq 0$ we have
\begin{equation}
\label{eq:exact-L2-rate}
\E\left|\int^t_0\mathcal{G}_\delta(\theta_{s}
\omega)ds-\omega(t)\right|^p = \frac{2^{p/2}\Gamma\left(\frac{p+1}{2}\right)}{\sqrt{\pi}} \delta^{pH}\Theta^{\frac{p}{2}}\left(\frac{t}{\delta}\right),
\end{equation}
where $\Gamma(z)$ denotes the Gamma function and $\Theta(x)$ is given by \eqref{eq:theta}.
\end{proposition}
\begin{proof}
Denote the error process by $[\Delta_\delta\omega]$, i.e.
\begin{eqnarray*}
[\Delta_\delta\omega](t) &=& \int^t_0\mathcal{G}_\delta(\theta_{s}\omega)ds-\omega(t) = \int^{t+\delta}_t\frac{\omega(s)-\omega(t)}{\delta}ds
-\int^{\delta}_0\frac{\omega(s)}{\delta}ds.
\end{eqnarray*}
For any $t>r$ we have
\begin{eqnarray*}
[\Delta_\delta\omega](t) - [\Delta_\delta\omega](r) &=& \int^{t+\delta}_t\frac{\omega(s)-\omega(t)}{\delta}ds - \int^{r+\delta}_r\frac{\omega(s)-\omega(r)}{\delta}ds \\
&=& \int_0^\delta \frac{\omega(s+t)-\omega(t)-\omega(s+r)+\omega(r)}{\delta}ds,
\end{eqnarray*}
and since fBm has stationary increments, it follows that 
\begin{eqnarray*}
[\Delta_\delta\omega](t) - [\Delta_\delta\omega](r) &=& \int_0^\delta \frac{\omega(s+t)-\omega(t)-\omega(s+r)+\omega(r)}{\delta}ds\\
&\overset{law}{=}&\int_0^\delta \frac{\omega(s+t-r)-\omega(t-r)-\omega(s)+\omega(0)}{\delta}ds\\
&=& [\Delta_\delta\omega](t-r),
\end{eqnarray*}
where $\overset{law}{=}$ denotes equality in distribution.
Thus the error process $x_t$ has stationary increments. It remains to prove \eqref{eq:exact-L2-rate}. First we observe that since $x_t$ is a Gaussian process and 
$$
\E|Z|^p = \frac{2^{p/2}\Gamma\left(\frac{p+1}{2}\right)}{\sqrt{\pi}}
$$
for $Z \sim N(0,1)$, it suffices to consider the case $p=2$. We compute
$$
\left[[\Delta_\delta\omega](t)\right]^2 = \int_0^\delta \int_0^\delta \frac{\left[\omega(s+t)-\omega(t)-\omega(s)\right]\left[\omega(u+t)-\omega(t)-\omega(u)\right]}{\delta^2}dsdu.
$$
After taking expectation, using Fubini's theorem, and plugging in \eqref{eq:fbm-covariance}, we can use some elementary manipulations to collect similar terms together and get
\begin{eqnarray*}
\E \left[[\Delta_\delta\omega](t)\right]^2 &=& \frac{1}{\delta^2}\int_0^\delta \int_0^\delta R_H(s+t,u+t)-R_H(s+t,t)-R_H(s+t,u) \\
&& -R_H(t,u+t) + R_H(t,t) + R_H(t,u) \\
 && -R_H(s,u+t) + R_H(s,t) + R_H(s,u) dsdu\\
&=& \frac{1}{2\delta^2}\int_0^\delta \int_0^\delta 2|t|^{2H} + 2|s|^{2H} + 2|u|^{2H}+|s-u+t|^{2H}+|s-u-t|^{2H} \\
&& -|s+t|^{2H}-|u+t|^{2H}-|t-u|^{2H} - |s-t|^{2H} - 2|s-u|^{2H}dsdu.
\end{eqnarray*}
Here we obtain by interchanging the roles of $s$ and $u$ that
$$
\int_0^\delta \int_0^\delta|s-u+t|^{2H}ds du = \int_0^\delta \int_0^\delta|s-u-t|^{2H}dsdu.
$$
Treating other terms similarly we obtain that
$$
\E \left[[\Delta_\delta\omega](t)\right]^2 = \frac{1}{\delta^2}\int_0^\delta \int_0^\delta t^{2H}+2u^{2H} + |s-u+x|^{2H}-(u+t)^{2H}-|u-t|^{2H} - |u-s|^{2H}ds du.
$$
Now the claim follows directly from a change of variables $u \mapsto \delta u$ and $s\mapsto \delta s$.
\end{proof}
We are now ready to formulate the following theorem that provides the rate of convergence in the space $W_{1,1-\beta}$. Up to multiplicative constants, our result is sharp.
\begin{theorem}
\label{theorem:WZ-fbm-convergence}
Let $H\in \left(\frac12,1\right)$ and $\beta\in(1-H,H)$. Then for any $p\geq 1$ there exists a constant $C=C(H,\beta,m,p)$, depending only on parameters $H,\beta,p$ and the number of dimensions $m$, such that for any $\delta<1$ we have
$$
\frac{1}{C}\delta^{H+\beta-1} \leq \left[\E\left\Vert\int_0^{\cdot}\mathcal{G}_\delta(\theta_{s}\omega)ds-\omega(\cdot)\right\Vert_{1,1-\beta}^p \right]^{\frac{1}{p}} \leq C \delta^{H+\beta-1}.
$$
\end{theorem}
\begin{proof}
Using the elementary inequality
$$
\max_{1\leq k\leq m} |x_k| \leq \sqrt{x_1^2 + \ldots x_m^2} \leq |x_1| + \ldots |x_n|,
$$
we obtain that it suffices to consider the one-dimensional case. Throughout the proof, we denote by $C$ any generic unimportant constant (depending on parameters $H,\beta,p$, and $m$) that may vary from line to line. As in the proof of Proposition \ref{pro:exact-L2-rate}, we use short notation $[\Delta_\delta\omega]$ for the error process.

We begin with the lower bound that is rather easy. Indeed, for any $0<s<t<T$ we have
$$
\Vert [\Delta_\delta\omega]\Vert_{1,1-\beta} \geq \frac{|[\Delta_\delta\omega](t)-[\Delta_\delta\omega](s)|}{(t-s)^{1+\beta}}.
$$
Consequently, Proposition \ref{pro:exact-L2-rate} yields
$$
\left[\E\left\Vert [\Delta_\delta\omega]\right\Vert_{1,1-\beta}^p \right]^{\frac{1}{p}} \geq C \delta^H \frac{\sqrt{\Theta\left(\frac{t-s}{\delta}\right)}}{(t-s)^{1+\beta}} = C \delta^{H+\beta-1}\frac{\sqrt{\Theta(v)}}{v^{1+\beta}},
$$
where $v=\frac{t-s}{\delta}$. As this holds for any $0<s<t<T$, it suffices to choose $s$ and $t$ such that 
$\frac{\sqrt{\Theta(v)}}{v^{1+\beta}}>0$. Consider next the upper bound. By Minkowski's inequality, we can study the two terms 
$$
\sup_{0<s<t<T}\frac{|[\Delta_\delta\omega](t) - [\Delta_\delta\omega](s)|}{(t-s)^{1-\beta}}
$$
and
\begin{equation}
\label{eq:integral-term}
\sup_{0<s<t<T}\int_s^t \frac{|[\Delta_\delta\omega](y)-[\Delta_\delta\omega](s)|}{|y-s|^{2-\beta}} \, dy
\end{equation}
separately. For the first one, a simple application of Borell-TIS inequality (see e.g. \cite[Theorem 2.1.1]{Adler-Taylor2007}) implies that, for any topological space $T$ and a centred Gaussian family $(\xi_t)_{t\in T}$, we have
$$
\left[\E \sup_{t\in T} |\xi_t|^p \right]^{\frac{1}{p}} \leq C\sup_{t\in T} \left[\E |\xi_t|^p \right]^{\frac{1}{p}}.
$$
Together with Proposition \ref{pro:exact-L2-rate}, this leads to an upper bound
\begin{eqnarray*}
&&\left[\E\left[\sup_{0<s<t<T}\frac{|[\Delta_\delta\omega](t) - [\Delta_\delta\omega](s)|}{(t-s)^{1-\beta}} \right]^p \right]^{\frac{1}{p}} \\
& \leq &C \left[\sup_{0<s<t<T} \E \left[\frac{|[\Delta_\delta\omega](t) - [\Delta_\delta\omega](s)|}{(t-s)^{1-\beta}} \right]^2 \right]^{\frac{1}{2}}\\
& = & C \sup_{0<s<t<T} \frac{\delta^H \sqrt{\Theta\left(\frac{t-s}{\delta}\right)}}{(t-s)^{1+\beta}} \\
&=& C \delta^{H+\beta-1} \sup_{0<v<\delta^{-1}T} \frac{\sqrt{\Theta\left(v\right)}}{v^{1+\beta}}. 
\end{eqnarray*}
Since $\beta>1-H$, Lemma \ref{lma:theta-properties} implies
\begin{equation*}
\sup_{0<v<\delta^{-1}T} \frac{\sqrt{\Theta\left(v\right)}}{v^{1+\beta}} 
\leq  \sup_{v>0} \frac{\sqrt{\Theta\left(v\right)}}{v^{1+\beta}} < \infty,
\end{equation*}
and hence we have obtained
\begin{equation}
\label{eq:boundary-finite}
\left[\E \left[\sup_{0<s<t<T}\frac{|[\Delta_\delta\omega](t) - [\Delta_\delta\omega](s)|}{(t-s)^{1-\beta}}\right]^p \right]^{\frac{1}{p}} \leq C\delta^{H+\beta-1}.
\end{equation}
Consider next the term \eqref{eq:integral-term}. By Proposition \ref{pro:exact-L2-rate}, $[\Delta_\delta\omega]$ has stationary increments. Thus
\begin{eqnarray*}
\sup_{0<s<t<T}\int_s^t \frac{|[\Delta_\delta\omega](y)-[\Delta_\delta\omega](s)|}{|y-s|^{2-\beta}} \, dy 
&\overset{law}{=} & \sup_{0<s<t<T}\int_s^t \frac{|[\Delta_\delta\omega](y-s)|}{|y-s|^{2-\beta}} \, dy \\
= \sup_{0<s<t<T} \int_0^{t-s} \frac{|[\Delta_\delta\omega](y)|}{y^{2-\beta}} \, dy 
&\leq & \int_0^T \frac{|[\Delta_\delta\omega](y)|}{y^{2-\beta}} \, dy,
\end{eqnarray*}
where $\overset{law}{=}$ denotes the equality in distribution. Now Minkowski's integral inequality together with Proposition \ref{pro:exact-L2-rate} yields
\begin{eqnarray*}
\left[\E \left[\int_0^T \frac{|[\Delta_\delta\omega](y)|}{y^{2-\beta}} \, dy\right]^p \right]^{\frac{1}{p}}
& \leq &\int_0^T \frac{\left[\E|[\Delta_\delta\omega](y)|^p\right]^{\frac{1}{p}}}{y^{2-\beta}} \, dy \\
\leq  C \int_0^T \frac{\delta^H \sqrt{\Theta\left(\frac{y}{\delta}\right)}}{y^{2-\beta}} \, dy 
& \leq & C \delta^{H+\beta-1} \int_0^\infty \frac{\sqrt{\Theta\left(y\right)}}{y^{2-\beta}} \, dy,
\end{eqnarray*}
where $\int_0^\infty \frac{\sqrt{\Theta\left(y\right)}}{y^{2-\beta}} \, dy < \infty$ by Lemma \ref{lma:theta-properties}. Hence we can conclude that
$$
\left[\E \left[\sup_{0<s<t<T}\int_s^t \frac{|[\Delta_\delta\omega](y)-[\Delta_\delta\omega](s)|}{|y-s|^{2-\beta}} \, dy \right]^p \right]^{\frac{1}{p}}\leq C \delta^{H+\beta-1}.
$$
Together with \eqref{eq:boundary-finite}, this concludes the proof.
\end{proof}

\section{Noise perturbations for differential equations driven by H\"older functions}
\label{sec4}
In this section we apply Theorem \ref{theorem:WZ-fbm-convergence} to study approximations of solutions to stochastic differential equations driven by fBm. However, In view of Remark 
\ref{remark:generalisation}, we present our main result of this section, Theorem \ref{thm4.1}, in a general form. The detailed analysis of the case of the fBm is postponed in Subsection \ref{subsec:WZ-SDE}.

Consider the following differential equation,
\begin{equation}\label{eq4.1}
du(t)=f(t,u(t))dt+\sigma(t,u(t))d\omega(t), ~~u(0)=x \in\mathbb{R}^n,
\end{equation}
 driven by $\omega \in C_\alpha$ for some $\alpha>\frac12$. To ensure the existence and uniqueness of the solution, we adopt the following slightly modified (cf. Remark \ref{remark:on-assumptions}) assumptions from \cite{NR2002} for
the coefficient
functions $f:\Omega\times[0,T]\times \mathbb{R}^n \to \mathbb{R}^n$ and $\sigma :\Omega\times[0,T]\times \mathbb{R}^n \to \mathbb{R}^{n\times m}$:

$(i)$ There exists a constant $M>0$ such that for $\forall x, y\in\mathbb{R}^n$, $\forall t\in[0,T]$
\begin{equation}\label{eq4.2}
|\sigma(t,x)-\sigma(t,y)|\leq M|x-y|,
\end{equation}
\begin{equation}\label{eq4.3}
|\sigma(t,x)-\sigma(s,x)|+|\partial_{x_i}\sigma(t,x)-\partial_{x_i}\sigma(s,x)|\leq M|t-s|.
\end{equation}
$(ii)$ For any $N>0$, there exist constants $M_N>0$ such that for $\forall |x|, |y|\leq N$, $\forall t\in[0,T]$
\begin{equation}\label{eq4.4}
|\partial_{x_i}\sigma (t,x)-\partial_{y_i}\sigma (t,y)|\leq M_N|x-y|.
\end{equation}
$(iii)$ There exists a constant $K_0>0$ and $\zeta\in[0,1]$ such that for $\forall x\in\mathbb{R}^n$, $\forall t\in[0,T]$
\begin{equation}\label{eq4.5}
|\sigma (t,x)|\leq K_0(1+|x|^\zeta).
\end{equation}
$(iv)$ For any $N>0$, there exists a constant $L_N>0$ such that for $\forall |x|, |y|\leq N$, $\forall t\in[0,T]$
\begin{equation}\label{eq4.6}
|f (t,x)-f (t,y)|\leq L_N|x-y|.
\end{equation}
$(v)$ There exist a constant $L_0>0$ and a function $b_0\in L^\rho(0,T;\mathbb{R}^n)$, where $\rho>2$, such that for $\forall x\in\mathbb{R}^n$, $\forall t\in[0,T]$
\begin{equation}\label{eq4.7}
|f (t,x)|\leq L_0|x|+b_0(t).
\end{equation}
\begin{remark}
\label{remark:on-assumptions}
The existence and uniqueness result of \cite{NR2002} hold under slightly more general assumptions compared to ours. That is, we have assumed Lipschitz continuity in \eqref{eq4.3} and \eqref{eq4.4} while in \cite{NR2002} these are replaced with H\"older continuity of suitable order. For the sake of simplicity and in order to avoid adding extra layers of parameters and complexity into the proof of Theorem \ref{thm4.1} that is already technical and rather lengthy, we work with the above set of assumptions. Note also that in condition (v) we do not allow $\rho=2$ which at first glimpse might look slightly different compared to $(H_2)$ of \cite{NR2002}. However, existence and uniqueness in the space $W_{\beta,\infty}$ follows from \cite[Theorem 2.1]{NR2002} provided that $\beta \in \left(1-\alpha,\frac12\right)$ and $\rho \geq \frac{1}{\beta}$. These conditions cannot be satisfied simultaneously if $\rho=2$. 
\end{remark}
Since $\omega$ is, in general, not smooth, Equation \eqref{eq4.1} is understood in integral form
\begin{equation}\label{eq4.8}
u(t)-x=\int^t_0f(s,u(s))ds+\int^t_0\sigma(s,u(s))d\omega(s).
\end{equation}
Note that here we have used $u(t)$ for the sake of notational simplicity, although $u$ depends also on the path of $\omega$, i.e. the solution is a flow $u(t,\omega)$. The integral in \eqref{eq4.8} is understood in the generalised Lebesgue-Stieltjes sense. Let $\beta \in \left(1-\alpha,\frac12\right)$ such that $\rho \geq \frac{1}{\beta}$. Then, by \cite[Theorem 2.1]{NR2002}, Equation (\ref{eq4.1}) has a unique solution in the space $C_{1-\beta} \subset W_{\beta,\infty}$ under the given conditions (\ref{eq4.2})-(\ref{eq4.7}). Moreover, by \cite[Proposition 5.1]{NR2002} the solution $u$ satisfies
\begin{equation}
\label{eq:solution-bound}
\Vert u \Vert_{\beta,\infty} \leq C_1 \exp\left(C_2\Vert \omega\Vert_{1,1-\beta}^\kappa\right),
\end{equation}
where $C_1$ and $C_2$ depend on the constants appearing in \eqref{eq4.2}-\eqref{eq4.7}, $T$, $\beta$, and $\alpha$. Finally, $\kappa$ depends solely on $\zeta$ and $\beta$. More precisely (cf. \cite[Proposition 5.1]{NR2002}), 
\begin{equation}
\label{eq:kappa}
\kappa = \begin{cases}
\frac{1}{1-2\beta},\quad \zeta = 1,\\
\frac{\zeta}{2-2\beta},\quad \frac{1-2\beta}{1-\beta} \leq \zeta < 1,\\
\frac{1}{1-\beta},\quad 0\leq \zeta < \frac{1-2\beta}{1-\beta}.
\end{cases}
\end{equation}
Let now $G_\delta \in C_\alpha$ be arbitrary approximation of the noise $\omega$ (for example, a smooth approximation). Then we approximate \eqref{eq4.8} with
\begin{equation}\label{eq4.10}
u_\delta(t)-x=\int^t_0f(s,u_\delta(s))ds+\int^t_0\sigma(s,u_\delta(s))dG_\delta(s).
\end{equation}
Since $G_\delta \in C_\alpha$ as well, Equation \eqref{eq4.10} admits a unique solution $u_\delta \in W_{\beta,\infty}$. It turns out that the solutions $u_\delta$ for \eqref{eq4.10} converge towards solution $u$ for \eqref{eq4.8} as long as $G_\delta$ converge towards the noise $\omega$ of  \eqref{eq4.8} in the
space $W_{1,1-\beta}$. 
\begin{theorem}\label{thm4.1}
Suppose that $f$ and $\sigma$ satisfies conditions \eqref{eq4.2}-\eqref{eq4.7}, and fix $\beta \in \left(1-\alpha,\frac12\right)$ such that $\rho \geq \frac{1}{\beta}$. Furthermore, let $u$ and, for a fixed family $\mathcal{D}$ and $\delta \in \mathcal{D}$, $u_\delta$ denote the unique solutions in $C_{1-\beta} \subset W_{\beta,\infty}$ to \eqref{eq4.8} and \eqref{eq4.10}, respectively. If $\sup_{\delta \in \mathcal{D}} \Vert G_\delta \Vert_{1,1-\beta} < \infty$, then there exists a constant $K$, independent of $\delta \in \mathcal{D}$, such that
\begin{equation}
\label{eq:sol-error}
\Vert u_\delta - u\Vert_{\beta,\infty} \leq K\|G_\delta-\omega\|_{1,1-\beta}.
\end{equation}
In particular, if $\Vert G_\delta - \omega\Vert_{1,1-\beta} \to 0$, then $
\Vert u_\delta - u\Vert_{\beta,\infty} \to 0.
$
\end{theorem}
\begin{proof}
Our aim is to apply Lemma \ref{lemma4.1}. For this we denote
\begin{equation}\label{eq4.11}
x_t=|u_\delta(t)-u(t)|
+\int^t_0\frac{|u_\delta(t)-u(t)-u_\delta(s)+u(s)|}{(t-s)^{\beta+1}}ds.
\end{equation}
It follows from \eqref{eq:solution-bound} that there exists large $N$ such that $\sup_{\delta}\Vert u_\delta\Vert_{\beta,\infty}\leq N$ and $\Vert u\Vert_{\beta,\infty} \leq N$. This allows us to apply localisation argument in order to apply conditions \eqref{eq4.2}-\eqref{eq4.7}. Throughout the proof, we denote by $K$ a generic constant that may vary from line to line. We stress that $K$ may depend on $N$ and the constants appearing in conditions \eqref{eq4.2}-\eqref{eq4.7}. $K$ may also depend on $\omega$, $T$, and $\beta$. However, $K$ is independent of $\delta$, although from time to time, we also apply bound $\Vert u_\delta\Vert_{\beta,\infty} \leq N$ and include this to the constant $K$, whenever confusion cannot arise.

We begin by computing
  \begin{align*}
x_t &=\left\{|u_\delta(t)-u(t)|
+\int^t_0\frac{|u_\delta(t)-u(t)-(u_\delta(s)-u(s))|}{(t-s)^{\beta+1}}ds\right\}\\
&\le\left\{\left|\int^t_0f(s,u_\delta(s))-f(s,u(s))ds\right|
+\int^t_0\frac{|\int^t_sf(r,u_\delta(r))-f(r,u(r))dr|}{(t-s)^{\beta+1}}ds\right\}\\
&\quad+\Bigg\{\left|\int^t_0\sigma(s,u_\delta(s))d G_\delta(s)
-\int^t_0\sigma(s,u(s))d\omega(s)\right|\\
&\quad+\int^t_0\frac{|\int^t_s\sigma(r,u_\delta(r))dG_\delta(r)
-\int^t_s\sigma(r,u(r))d\omega(r)|}{(t-s)^{\beta+1}}ds\Bigg\}\\
&\doteq I_1(t)+I_2(t).
\end{align*}
Note also that
$$
\|u_\delta-u\|_{\beta,\infty} = \sup_{t\in[0,T]} (I_1(t)+I_2(t)).
$$
For $I_1(t)$, it follows from (\ref{eq4.6}) that
\begin{align*}
  I_1&\le L_N \int^t_0|u_\delta(s)-u(s)|ds
+L_N\int^t_0\frac{\int^t_s  |u_\delta(r)-u(r)| dr}{(t-s)^{\beta+1}}ds\\
 &=L_N \int^t_0|u_\delta(s)-u(s)|ds
+L_N\int^t_0\int^r_0 \frac{ |u_\delta(r)-u(r)|}{(t-s)^{\beta+1}}ds dr\\
&\le L_N \int^t_0|u_\delta(s)-u(s)|ds
+\frac{L_N}{\beta}\int^t_0 \frac{ |u_\delta(r)-u(r)|}{(t-r)^{\beta}}dr\\
&\le L_N \left(T^{2\beta}+\frac{T^\beta}{\beta}\right)t^{2\beta}\int^t_0 \frac{ |u_\delta(r)-u(r)|}{(t-r)^{2\beta}r^{2\beta}}dr.
\end{align*}
By the notation of (\ref{eq4.11}), we have thus observed
\begin{equation}\label{eq4.12}
I_1(t)\le Kt^{2\beta}\int^t_0 (t-s)^{-2\beta}s^{-2\beta}x_sds.
\end{equation}

For $I_2(t)$, we observe first that, for any $s,t \in [0,T]$,
\begin{align*}
&\left|\int_s^t\sigma(r,u_\delta(r))dG_\delta(r)
-\int_s^t\sigma(r,u(r))d\omega(r)\right|\\
&\leq\left|\int_s^t\sigma(r,u_\delta(r))d(G_\delta(r)-\omega(r))\right|
+\left|\int_s^t(\sigma(r,u_\delta(r))-\sigma(r,u(r)))d\omega(r)\right|.
\end{align*}
For the first term we get, by \eqref{eq:Holder-for-integral}, that
\begin{equation}\label{eq4.13}
\left|\int_s^t\sigma(r,u_\delta(r))d(G_\delta(r)-\omega(r))\right|\leq\frac{|t-s|^{1-\beta}}{\Gamma(1-\beta)\Gamma(\beta)}\|G_\delta-\omega\|_{1,1-\beta}\|\sigma(\cdot,u_\delta)\|_{\beta,\infty}.
\end{equation}
Here
$$
\Vert \sigma(\cdot,u_\delta)\Vert_{\beta,\infty} \leq M\left(T+\frac{T^{1-\beta}}{1-\beta}\right)+|\sigma(0,0)|+\Vert u_\delta\Vert_{\beta,\infty}\leq K
$$
by Lipschitz continuity of $\sigma$ and the fact that $\Vert u_\delta\Vert_{\beta,\infty}\leq N$.
For the second term, we apply \eqref{eq:basic-bound} and Lemma 7.1 of \cite{NR2002} to get
\begin{align*}
&\left|\int_0^t \sigma(s,u_\delta(s))-\sigma(s,u(s))d\omega(s)\right|\\
&\quad\leq K\bigg[\int_0^t\frac{|\sigma(s,u_\delta(s))-\sigma(s,u(s))|}{s^\beta}ds\\
&\qquad+\beta\int_0^t\int_0^s\frac{|\sigma(s,u_\delta(s))-\sigma(s,u(s))|
-|\sigma(r,u_\delta(r))-\sigma(r,u(r))|}{(s-r)^{\beta+1}}drds\bigg]\\
&\quad\leq K\bigg[\int_0^t\frac{|u_\delta(s)-u(s)|}{s^\beta}ds+\int_0^t\int_0^s\frac{|u_\delta(s)-u(s)
-u_\delta(r)+u(r)|}{(s-r)^{\beta+1}}drds\\
&\qquad+\int_0^t\int_0^s\frac{|u_\delta(s)-u(s)|(|u_\delta(s)-u_\delta(r)|+|u(s)-u(r)|)}
{(s-r)^{\beta+1}}drds\\
&\qquad+\int_0^t\int_0^s\frac{|u_\delta(s)-u(s)|}
{(s-r)^{\beta}}drds\bigg]\\
&\quad\leq K\bigg[T^\beta t^{2\beta}\int_0^t(t-s)^{-2\beta}s^{-2\beta}|u_\delta(s)-u(s)|ds\\
&\qquad+T^{2\beta} t^{2\beta}\int_0^t
(t-s)^{-2\beta}s^{-2\beta}\int_0^s\frac{|u_\delta(s)-u(s)
-u_\delta(r)+u(r)|}{(s-r)^{\beta+1}}drds\\
&\qquad+T^{2\beta} t^{2\beta}(\|u_\delta\|_{\beta,\infty}+\|u\|_{\beta,\infty})\int_0^t(t-s)^{-2\beta}s^{-2\beta}|u_\delta(s)-u(s)|ds\\
&\qquad+\frac{T^{1+\beta} t^{2\beta}}{1-\beta}\int_0^t(t-s)^{-2\beta}s^{-2\beta}|u_\delta(s)-u(s)|ds\bigg].
\end{align*}
Since we have $\Vert u_\delta\Vert_{\beta,\infty}\leq N$ and $\Vert u\Vert_{\beta,\infty}\leq N$, it follows that
\begin{equation}\label{eq4.14}
\left|\int^t_0\sigma(s,u_\delta(s))-\sigma(s,u(s))d\omega(s)\right|\leq Kt^{2\beta}\int^t_0(t-s)^{-2\beta}s^{-2\beta}x_sds.
\end{equation}
Similarly, we have
\begin{align*}
&\int_0^t\frac{\left|\int_s^t\sigma(r,u_\delta(r))dG_\delta(r)
-\int_s^t\sigma(r,u(r))d\omega(r)\right|}{(t-s)^{\beta+1}}ds\\
&\quad \leq K\int_0^t\frac{|t-s|^{1-\beta}\|G_\delta-\omega\|_{1,1-\beta}\|\sigma(\cdot,u_\delta)\|_{\beta,\infty}}{(t-s)^{\beta+1}}ds\\
&\quad +\int_0^t\frac{\left|\int_s^t(\sigma(r,u_\delta(r))
-\sigma(r,u(r)))d\omega(r)\right|}{(t-s)^{\beta+1}}ds\\
&\quad \leq K\|G_\delta-\omega\|_{1,1-\beta} \\
&\quad +\int_0^t\frac{\left|\int_s^t(\sigma(r,u_\delta(r))
-\sigma(r,u(r)))d\omega(r)\right|}{(t-s)^{\beta+1}}ds.
\end{align*}
Here
\begin{align*}
&\int_0^t\frac{\left|\int_s^t(\sigma(r,u_\delta(r))
-\sigma(r,u(r)))d\omega(r)\right|}{(t-s)^{\beta+1}}ds\\
&\quad\leq K\int_0^t\frac{\int_s^t\frac{|\sigma(r,u_\delta(r))
-\sigma(r,u(r))|}{(r-s)^\beta}dr}{(t-s)^{\beta+1}}ds\\
&\qquad+K\int_0^t\frac{\int_s^t\int_s^r\frac{|\sigma(r,u_\delta(r))
-\sigma(r,u(r))-\sigma(y,u_\delta(y))+\sigma(y,u(y))|}{(r-y)^{\beta+1}}dydr}{(t-s)^{\beta+1}}ds\\
&\quad\leq K\int_0^t\frac{\int_s^t\frac{|u_\delta(r)
-u(r)|}{(r-s)^\beta}dr}{(t-s)^{\beta+1}}ds
+K\int_0^t\frac{\int_s^t\int_s^r\frac{|u_\delta(r)
-u(r)-u_\delta(y)+u(y)|}{(r-y)^{\beta+1}}dydr}{(t-s)^{\beta+1}}ds\\
&\qquad+K\int_0^t\frac{\int_s^t\int_s^r\frac{|u_\delta(r)-u(r)|
(|u_\delta(r)-u_\delta(y)|+|u(r)-u(y)|)}{(r-y)^{\beta+1}}dydr}{(t-s)^{\beta+1}}ds\\
&\qquad+K\int_0^t\frac{\int_s^t\int_s^r\frac{|u_\delta(r)-u(r)|}{(r-y)^{\beta}}dydr}{(t-s)^{\beta+1}}ds.
\end{align*}

Let $s=r-(t-r)z$. For the first term, we have
\begin{align*}
&\int_0^t\frac{\int_s^t\frac{|u_\delta(r)-u(r)|}{(r-s)^\beta}dr}{(t-s)^{\beta+1}}ds\\
&\quad=\int_0^t\int_0^r(t-s)^{-\beta-1}\frac{|u_\delta(r)-u(r)|}{(r-s)^\beta}dsdr\\
&\quad=\int_0^t\int_0^{\frac{r}{t-r}}(t-r)^{-2\beta}(1+z)^{-\beta-1}z^{-\beta}|u_\delta(r)-u(r)|dzdr\\
&\quad\leq B(2\beta,1-\beta)\int_0^t(t-r)^{-2\beta}|u_\delta(r)-u(r)|dr\\
&\quad\leq K t^{2\beta}\int_0^t(t-r)^{-2\beta}r^{-2\beta}|u_\delta(r)-u(r)|dr,
\end{align*}
where
\begin{equation*}
B(p,q)=\int_0^\infty(1+t)^{-x-y}t^{y-1}dt.
\end{equation*}
This yields 
\begin{equation}\label{eq4.15}
\int_0^t\frac{\int_s^t\frac{|u_\delta(r)-u(r)|}{(r-s)^\beta}dr}{(t-s)^{\beta+1}}ds\leq Kt^{2\beta}\int^t_0(t-s)^{-2\beta}s^{-2\beta}x_sds.
\end{equation}
For the second term, we have
\begin{align*}
&\int_0^t\frac{\int_s^t\int_s^r\frac{|u_\delta(r)
-u(r)-u_\delta(y)+u(y)|}{(r-y)^{\beta+1}}dydr}{(t-s)^{\beta+1}}ds\\
&\quad=\int_0^t\int_0^r\int_0^y(t-s)^{-\beta-1}\frac{|u_\delta(r)-u(r)-u_\delta(y)+u(y)|}
{(r-y)^{\beta+1}}dsdydr\\
&\quad\leq\int_0^t\int_0^r\frac{1}{\beta}(t-y)^{-\beta}
\frac{|u_\delta(r)-u(r)-u_\delta(y)+u(y)|}{(r-y)^{\beta+1}}dydr\\
&\quad\leq KT^\beta t^{2\beta}\int_0^t(t-r)^{-2\beta}r^{-2\beta}\int_0^r
\frac{|u_\delta(r)-u(r)-u_\delta(y)+u(y)|}{(r-y)^{\beta+1}}dydr,
\end{align*}
leading to
\begin{equation}\label{eq4.16}
\int_0^t\frac{\int_s^t\int_s^r\frac{|u_\delta(r)
-u(r)-u_\delta(y)+u(y)|}{(r-y)^{\beta+1}}dydr}{(t-s)^{\beta+1}}ds\leq Kt^{2\beta}\int^t_0(t-s)^{-2\beta}s^{-2\beta}x_sds.
\end{equation}
For the third term, we obtain
\begin{align*}
&\int_0^t\frac{\int_s^t\int_s^r\frac{|u_\delta(r)-u(r)|
(|u_\delta(r)-u_\delta(y)|+|u(r)-u(y)|)}{(r-y)^{\beta+1}}dydr}{(t-s)^{\beta+1}}ds\\
&\quad\leq (\|u_\delta\|_{\beta,\infty}+\|u\|_{\beta,\infty})\int_0^t\frac{\int_s^t|u_\delta(r)
-u(r)|dr}{(t-s)^{\beta+1}}ds\\
&\quad\leq K\int_0^t\int_0^r(t-s)^{-\beta-1}|u_\delta(r)-u(r)|dsdr\\
&\quad\leq K\int_0^t(t-r)^{-\beta}|u_\delta(r)-u(r)|dr\\
&\quad\leq KT^\beta t^{2\beta}\int_0^t(t-r)^{-2\beta}r^{-2\beta}|u_\delta(r)-u(r)|dr,
\end{align*}
and hence
\begin{equation}\label{eq4.17}
\int_0^t\frac{\int_s^t\int_s^r\frac{|u_\delta(r)-u(r)|
(|u_\delta(r)-u_\delta(y)|+|u(r)-u(y)|)}{(r-y)^{\beta+1}}dydr}{(t-s)^{\beta+1}}ds
\leq Kt^{2\beta}\int^t_0(t-s)^{-2\beta}s^{-2\beta}x_sds.
\end{equation}
Finally, for the fourth term we obtain
\begin{align*}
&\int_0^t\frac{\int_s^t\int_s^r\frac{|u_\delta(r)-u(r)|}{(r-y)^{\beta}}dydr}{(t-s)^{\beta+1}}ds\\
&\quad\leq \frac{T^{1-\beta}}{1-\beta}\int_0^t\frac{\int_s^t|u_\delta(r)
-u(r)|dr}{(t-s)^{\beta+1}}ds\\
&\quad\leq K\int_0^t\int_0^r(t-s)^{-\beta-1}|u_\delta(r)-u(r)|dsdr\\
&\quad\leq K\int_0^t(t-r)^{-\beta}|u_\delta(r)-u(r)|dr\\
&\quad\leq KT^\beta t^{2\beta}\int_0^t(t-r)^{-2\beta}r^{-2\beta}|u_\delta(r)-u(r)|dr,
\end{align*}
so
\begin{equation}\label{eq4.18}
\int_0^t\frac{\int_s^t\int_s^r\frac{|u_\delta(r)-u(r)|}{(r-y)^{\beta}}dydr}{(t-s)^{\beta+1}}ds
\leq Kt^{2\beta}\int^t_0(t-s)^{-2\beta}s^{-2\beta}x_sds.
\end{equation}
Consequently, by combining the estimates (\ref{eq4.13})-(\ref{eq4.18}) we obtain an estimate
\begin{equation}\label{eq4.19}
I_2(t) \leq K\|G_\delta-\omega\|_{1,1-\beta} + Kt^{2\beta}\int^t_0 (t-s)^{-2\beta}s^{-2\beta}x_sds.
\end{equation}
With (\ref{eq4.12}) this leads to
\begin{equation*}
\begin{split}
x_t &\leq I_1(t)+I_2(t) \\
& \leq K\|G_\delta-\omega\|_{1,1-\beta} +
Kt^{2\beta}\int^t_0 (t-s)^{-2\beta}s^{-2\beta}x_sds,
\end{split}
\end{equation*}
and thus we may apply Lemma \ref{lemma4.1} to get
$$
\Vert u_\delta - u\Vert_{\beta,\infty} = \sup_{t\in[0,T]} x_t \leq K\|G_\delta-\omega\|_{1,1-\beta}.
$$
This shows \eqref{eq:sol-error} and completes the proof.
\end{proof}

\subsection{Wong-Zakai approximations for SDEs driven by fBm}
\label{subsec:WZ-SDE}
Consider the following stochastic differential equations (SDEs) driven by fBm
\begin{equation}\label{eq:sde-fbm}
du(t)=f(t,u(t))dt+\sigma(t,u(t))dW^H(t), ~~u(0)=x \in\mathbb{R}^n,
\end{equation}
where $H>\frac12$,
and the corresponding Wong--Zakai approximation
\begin{equation}\label{eq:sde-fbm-approx}
\dot{u}_\delta(t)=f(t,{u}_\delta(t))dt+\sigma(t,{u}_\delta(t))\mathcal {G}_\delta(\theta_t\omega),~~ {u}_\delta(0)=x \in\mathbb{R}^n,
\end{equation}
where $\mathcal {G}_\delta(\theta_t\omega)$ is given by (\ref{equ4.1}). We denote
\begin{equation*}
  G_\delta(t,\omega):=( G_\delta(t,\omega_1),\cdots, G_\delta(t,\omega_m)),
\end{equation*}
where for each $1\le j\le m$
\begin{equation*}
G_\delta(t,\omega_j):=\int_0^t\mathcal {G}_\delta(\theta_s\omega_j)ds.
\end{equation*}
Then Equation (\ref{eq:sde-fbm-approx}) can be rewritten as
\begin{equation}\label{eq:sde-fbm-approx2}
\dot{u}_\delta=f(t,{u}_\delta)dt+\sigma(t,{u}_\delta)\dot{G}_\delta(t,\omega),~~ u_\delta(0)=x, x\in\mathbb{R}^n,
\end{equation}
and equations (\ref{eq:sde-fbm-approx}) and (\ref{eq:sde-fbm}) can be written in integral forms as
\begin{equation}\label{eq:sde-fbm-approx-integral}
u_\delta(t,\omega)-x=\int^t_0f(s,u_\delta(s,\omega))ds+\int^t_0\sigma(s,u_\delta(s,\omega))\dot{G}_\delta(s,\omega)ds,
\end{equation}
and
\begin{equation}\label{eq:sde-fbm-integral}
u(t,\omega)-x=\int^t_0f(s,u(s,\omega))ds+\int^t_0\sigma(s,u(s,\omega))d\omega(s),
\end{equation}
where we have identified $W^H$ with $\omega$.
Together with (\ref{equ3.3}), conditions \eqref{eq4.2}-\eqref{eq4.7} give the existence and uniqueness of solution of random ordinary differential Equation (\ref{eq:sde-fbm-approx2}) by standard arguments. On the other hand, solution to \eqref{eq:sde-fbm-approx2}
is the solution to \eqref{eq:sde-fbm-approx-integral}, where the integral
$$
\int^t_0\sigma(s,u_\delta(s,\omega))\dot{G}_\delta(s,\omega)ds = \int^t_0\sigma(s,u_\delta(s,\omega))dG_\delta(s,\omega)
$$
can also be understood in the generalised Lebesgue-Stieltjes sense. Thus, again by \cite[Theorem 2.1]{NR2002}, Equation \eqref{eq:sde-fbm-approx-integral} has a unique solution in the space $C_{1-\beta} \subset W_{\beta,\infty}$, where $\beta\in \left(1-H,\frac12\right)$.

The following theorem provides us the fact that one can approximate solutions to SDEs driven by fBm by studying equations driven by our smooth stationary approximation. The claim follows essentially by repeating arguments presented in the proof of Theorem \ref{theorem:WZ-fbm-convergence} and in \cite{NR2002}. For this reason we only sketch the main ideas and omit the details.
\begin{theorem}
Suppose that $f$ and $\sigma$ satisfies conditions \eqref{eq4.2}-\eqref{eq4.7}, and fix $\beta \in \left(1-H,\frac12\right)$ such that $\rho\geq \frac{1}{\beta}$. 
Let $u_\delta$ and $u$ be the (unique in $C_{1-\beta}$) solution to \eqref{eq:sde-fbm-approx-integral} and \eqref{eq:sde-fbm-integral}, respectively. Then, as $\delta \to 0$, 
\begin{enumerate}
\item We have $\Vert u_\delta - u\Vert_{\beta,\infty} \to 0$ in probability. 
\item If in addition $\alpha < \frac{2-\zeta}{4}$ and the constants and the function $b_0$ in \eqref{eq4.2}-\eqref{eq4.7} are deterministic and can be chosen independently of $N$, then there exists a constant $K$, independent of $\delta$, such that
$$
\left[\E\Vert u_\delta - u\Vert_{\beta,\infty}^p \right]^{\frac{1}{p}} \leq K \delta^{H+\beta-1}.
$$
\end{enumerate}
\end{theorem}
\begin{proof}
The first claim follows immediately by combining statements of Theorem \ref{theorem:WZ-fbm-convergence} and Theorem \ref{thm4.1} together with the fact that convergence in $L_p(\Omega)$ implies convergence in probability. Similarly, following the proof of Theorem \ref{thm4.1} and applying \eqref{eq:solution-bound} gives us
\begin{equation*}
\Vert u_\delta - u\Vert_{\beta,\infty} \leq C e^{C \left[\Vert u_\delta\Vert_{1,1-\beta}^\kappa + \Vert u\Vert_{1,1-\beta}^\kappa\right]}\Vert\Vert u_\delta - u\Vert_{\beta,\infty},
\end{equation*}
where $\kappa$ is given by \eqref{eq:kappa}. Estimating the norm $\Vert \cdot \Vert_{1,1-\beta}$ from above by the H\"older norm, obtaining that given assumptions imply $\kappa<2$, repeating the arguments of \cite{NR2002} or \cite{Azmoodeh-Sottinen-Viitasaari-Yazigi-2014} to estimate exponential moments of the H\"older norm, and using H\"older inequality together with Theorem \ref{theorem:WZ-fbm-convergence} proves the claim. The repetition of the technical details are left to the reader.
\end{proof}
\begin{remark}
\label{rem:rate}
As expected, the sharp rate of convergence in Theorem \ref{theorem:WZ-fbm-convergence} translates into a similar rate of convergence for the solutions. While here we have only provided an upper bound, the obtained rate is the best one can hope in the general setting that covers all the possible choices of the coefficients $f$ and $\sigma$. 
\end{remark}
\section{Discussions}
\label{sec:discussions}
In this article we have introduced Wong--Zakai approximations of the fractional noise and studied its convergence properties. Our approximation is valid on the full range $H\in(0,1)$ of the Hurst parameter. As an application, we proved that, for the case $H>\frac12$, the solutions of the approximating differential equations converge, in the norm $\Vert \cdot\Vert_{\beta,\infty}$, towards the original solution. The only needed feature for the convergence of approximating solutions to differential equations is that the approximation of the noise converge in $W_{1,1-\beta}$. Indeed, on top of technical computations the only additional ingredient in the proof of Theorem \ref{thm4.1} was a deterministic Gronwall type Lemma \ref{lemma4.1}. 

In the particular case of the fractional Brownian motion, convergence of the approximation together with a sharp rate of convergence can be seen from Lemma \ref{lma:theta-properties} and Proposition \ref{pro:exact-L2-rate} whose proof uses the fact that the underlying process is the fBm. This in turn translates into sharp result, provided in Theorem \ref{theorem:WZ-fbm-convergence}, on the rate of convergence in the space $W_{1,1-\beta}$. In addition, proof of Theorem \ref{theorem:WZ-fbm-convergence} also applies Gaussianity and stationarity of the increments. However, as pointed out already in Remark \ref{remark:generalisation}, the convergence in $L_p(\Omega)$ follows immediately if
\begin{equation}
\label{eq:discuss-holder}
\E|\omega(u)-\omega(v)| \leq C|u-v|^{H},
\end{equation}
and
$
[\Delta_\delta\omega](t) = \int_0^{t}\mathcal{G}_\delta(\theta_{s}\omega)ds-\omega(t)
$
satisfies the hypercontractivity property
\begin{equation}
\label{eq:discuss-hyper}
\E \left|[\Delta_\delta\omega](t)-[\Delta_\delta\omega](s)\right|^p \leq C_p [\E \left|[\Delta_\delta\omega](t)-[\Delta_\delta\omega](s)\right| ]^{p}.
\end{equation}
The convergence in $W_{1,1-\beta}$ on the other hand could be based on Garsia-Rodemich-Rumsey lemma \cite{GRR1970} which provides an inequality
$$
|f(t)-f(s)|^p \leq C|t-s|^{\alpha p-1}\int_0^T\int_0^T \frac{|f(x)-f(y)|^p}{|x-y|^{\alpha p +1}}dxdy,
$$
valid for all continuous functions $f$ on $[0,T]$, all $p\geq 1$, and $\alpha > \frac{1}{p}$. Using this one can follow the proof of \cite[Lemma 7.4]{NR2002} or \cite[Theorem 1]{Azmoodeh-Sottinen-Viitasaari-Yazigi-2014} and, assuming hypercontractivity, eventually obtain convergence of the moments of $\Vert [\Delta_\delta\omega]\Vert_{1,1-\beta}$. Unfortunately however, while providing the desired convergence in the space $W_{1,1-\beta}$, this approach does not yield (in any straightforward manner) good bounds for the rate of convergence. Indeed, even in the case of fBm, one obtains that the $p$:th moment is proportional to $\delta^{\gamma}$, where $\gamma\in(0,1)$ can be chosen arbitrarily and $p$ has to be chosen large enough. Thus the obtained rate of convergence in $L_p(\Omega)$ is proportional to $\delta^{\frac{\gamma}{p}}$ which in turn means that one obtains worse rate for higher moments. This is significantly worse than the sharp rate $\delta^{H+\beta-1}$, valid for all $p\geq 1$, provided in Theorem \ref{theorem:WZ-fbm-convergence}. 

On the positive side, Garsia-Rodemich-Rumsey inequality still provides the required convergence (despite poor rate of convergence) with very modest assumptions, namely hypercontractivity and H\"older continuity, on the underlying process $\omega$. This fact extends our results to cover a rather rich class of stochastic processes. For example, all H\"older continuous Gaussian processes fall into this category, see \cite[Theorem 1]{Azmoodeh-Sottinen-Viitasaari-Yazigi-2014}. Interesting non-Gaussian examples include $k$th order Hermite processes. They share many properties with the fBm including covariance structure, H\"older continuity, and self-similarity, though they are not Gaussian processes but instead belong to the $k$th chaos (see, e.g. \cite{maejima-tudor} for definition and details). In this case \eqref{eq:discuss-holder} and \eqref{eq:discuss-hyper} are both valid, and consequently the results of this paper provide stationary approximations to Hermite processes such that the corresponding solutions of stochastic differential equations converge as well. It is worth to note that, to the best of our knowledge, stochastic differential equations driven by Hermite processes have not been extensively studied in the literature.

Finally we want to emphasize once more that while the results of Section \ref{sec4} require $H>\frac12$, the stationary Wong--Zakai approximation and the exact $L_p(\Omega)$-error provided in Theorem \ref{theorem:WZ-fbm-convergence} are valid for all $H\in(0,1)$. This provides a natural question for future research on the convergence of solutions in the region $H\in \left(0,\frac12\right)$, in which case rough path theory or some other method of integration has to be considered. 

\section*{Acknowledgements} This research was partially supported by the National Natural Science Foundation of China (No. 11871225,11501216), and the Fundamental Research Funds for the Central Universities (No. 2018MS58).

\bibliographystyle{plain}
\bibliography{bibli_wz}

\end{document}